\documentclass{article}
\usepackage[utf8]{inputenc}

\usepackage[english]{babel}

\usepackage{amsmath}
\usepackage{amsfonts}
\usepackage{amsthm}
\usepackage{amssymb}

\usepackage{graphicx}
\usepackage{caption}
\usepackage{subcaption}
\usepackage{wrapfig}

\usepackage{verbatim} 

\usepackage{authblk}
\usepackage{hyperref}

\newtheorem{theorem}{Theorem}[section]
\newtheorem{lemma}{Lemma}[section]

\DeclareMathOperator{\Res}{Res}

\title{On Algebraic-Geometry Approach to Ribaucour Transformations}
\date{}
\author{Evgeniy Glukhov\\ e-mail: \href{mailto:evgeniy.glukhov.eg@gmail.com}{evgeniy.glukhov.eg@gmail.com} }
\affil{L. D. Landau Institute for Theoretical Physics
of Russian Academy of Sciences}

\begin{document}

\maketitle

\begin{abstract}
We develop the idea of using an algebraic-geometry approach to classical differential geometry problems. Consider an orthogonal net constructed according to algebraic-geometric data we obtain a set of smooth orthogonal nets that are Ribaucour transformations of the initial orthogonal net.  
\end{abstract}

\section{Introduction}

Krichever, in his work \cite{Krichever}, introduced a remarkable approach for constructing curvilinear orthogonal coordinates in Euclidean space according to algebraic-geometric data. This method was extended to construct orthogonal coordinates in various Riemannian spaces in the subsequent works \cite{berdinsky}, \cite{Glukhov}. In this paper, we study transformations of initial algebraic-geometric data that result in Ribaucour transformations of orthogonal nets.
  
A smooth map $x(u):U\rightarrow\mathbb{R}^{n+N}$, $U\subset\mathbb{R}^n$, $u=(u^1,\dots,u^n)$, $x=(x^1,\dots x^{n+N}),$ is called \textit{an $n$-dimensional orthogonal net of codimension $N$} if
\begin{enumerate}
\item $\partial_{i}\partial_j x(u) = c_{ij}^i(u) \partial_i x(u)+c_{ij}^j(u) \partial_j x(u),\quad i\neq j,$
\item $\partial_i x(u)\cdot  \partial_j x(u) = 0,\quad i\neq j,$
\end{enumerate}
where $x\cdot y$ is the standard Euclidean scalar product in $\mathbb{R}^{n+N}.$  
  
If $N=0$, then orthogonal nets are curvilinear orthogonal coordinates in $\mathbb{R}^n$. Krichever was the first who proposed algebraic-geometry methods for constructing orthogonal coordinate systems in $\mathbb{R}^n$. In his work \cite{Krichever}, such systems were described in terms of the $\theta$-functions of smooth algebraic curves.

If $N=1$, then orthogonal nets are hypersurfaces that parametrized by curvature line coordinates. An important example of such hypersurfaces is a constant curvature hypersurface with orthogonal coordinates on it. An extension of Krichever's method to constant curvature hypersurfaces was developed in the work \cite{berdinsky}.  

For orthogonal nets of an arbitrary codimension $N$, a generalization of Krichever's construction was proposed in \cite{Glukhov}. The geometry of normal bundle of an orthogonal net according to algebraic-geometric data was also studied in this work.

All the formulae for the described constructions are expressed in terms of the $\theta$-function of a smooth algebraic curve. Mironov and Taimanov in \cite{Mironov} modified Krichever's construction for singular algebraic curves. They constructed curvilinear orthogonal coordinates by algebraic-geometry methods considering a singular algebraic curve that results in elementary functions description.

The geometry of orthogonal nets is closely connected to the theory of integrable systems. Integrability of the Lam\'e equations (that describe flat diagonal metrics) via the dressing method was proved by Zakharov in \cite{zakharov}. Further application of diagonal metrics includes, for example, works on hydrodynamic type systems, Frobenius manifolds, and the theory of compatible metrics (see \cite{Tsarev}, \cite{Mokhov_1990}, \cite{ferapontov}, \cite{dubrovin}, \cite{Mokhov_2000}, \cite{Mokhov_2004}, \cite{Mokhov_2010}). 

We study Ribaucour transformations of orthogonal nets that are obtained by algebraic-geometry methods. The theory of Ribaucour transformation was developed by classics, including Bianchi, and it has various extensions (see \cite{Tojeiro}, \cite{Hertchin}). 

Recall here an analytical description for Ribaucour transformations of orthogonal nets. A pair of $n$-dimensional orthogonal nets $x(u)$, $\tilde{x}(u)$ is called \textit{a Ribaucour pair} if for some functions $\lambda_1(u),\dots, \lambda_n(u)$ we have
\begin{equation}
\label{RibTrans}
\partial_i \tilde{x}(u)=\lambda_i(u) \left(\partial_i x(u) - 2 \dfrac{\partial_i x(u)\cdot\delta x(u)}{\delta x(u)\cdot\delta x(u)}\delta x(u)\right),
\end{equation}
where $\delta x(u) = \tilde{x}(u)-x(u).$ The orthogonal net $\tilde{x}(u)$ is called \textit{a Ribaucour transformation} of $x(u)$.

The aim of this paper is to prove that certain transformations of algebraic-geometric data result in Ribaucour transformations of corresponding orthogonal nets. So we get a class of algebraic-geometric Ribaucour pairs within the proposed approach. It is an open question if any Ribaucour pair is in this class, and this will be the object of further studies. 

Ribaucour transformations possess an important permutability property. It states that for any orthogonal net $x(u)$ and its two Ribaucour transformations $x_1(u)$ and $x_2(u)$ there exists a fourth orthogonal net $x_{12}(u)$ that is a Ribaucour transformation of $x_1(u)$ and $x_2(u)$. A set of all four orthogonal nets is called \textit{a Bianchi quadrilateral}. Moreover, one can choose 1-parameter family of simultaneous Ribaucour transformations $x_{12}$ if corresponding points of $x,\,x_1,\,x_2,\,x_{12}$ are concircular \cite{bianchi}.

The permutability property has extensions to a greater number of initial Ribaucour transformations \cite{Ganzha}, \cite{BURSTALL}, \cite{Tojeiro2}. It plays a key role in the theory of discrete orthogonal nets \cite{bobenko}. Let us also note that there is an algebraic-geometry approach for discrete orthogonal nets \cite{Akhmetshin}.

To be definite, recall a general permutability property for Ribaucour transformations. Given $l$ Ribaucour transformations of an orthogonal net, a generic choice of $\dfrac{l(l-1)}{2}$ Ribaucour transformations of every pair of them results in unique Bianchi $l$-cube, i.e. $2^l$ orthogonal nets with the combinatorics of $l$-cube defined by Ribaucour pairs. We obtain an example of Bianchi $l$-cube for the initial orthogonal net $x(u)$ that is determined by a certain transformation of algebraic-geometric data.

\section{Algebraic-Geometry Approach to Orthogonal Nets}

We define here $n$-dimensional algebraic-geometric orthogonal nets of codimension $N$. The construction is based on \cite{Glukhov} and is a natural development of the Krichever construction \cite{Krichever}. 

Consider a smooth complex curve $\Gamma$ of genus $g$, fix $l\in\mathbb{N}$ and three divisors on the curve: $P_1+\dots+P_n,$  $\gamma=\gamma_1+\dots+\gamma_{g+l+N-1},$ $R=R_1+\dots+R_{l+N}$,  and introduce a local parameter $k_j^{-1}$ in the neighborhood of every point $P_j$, $k_j^{-1}(P_j) = 0$. Denote all the algebraic-geometric data by $S$:
$$S=\{\Gamma,\, g;\; \{P_j,\,k_j^{-1}\}_{j=1}^{n};\; R_1,\dots, R_{l+N};\gamma_1,\dots,\gamma_{g+l+N-1}\}.$$

A function $\psi(u^1,\dots,u^n;\;Q | d)$, $Q\in\Gamma$, $d=(d_1,\dots,d_{l+N})\in\mathbb{R}^{l+N},$ with the following analytical properties:
\begin{enumerate}
\item in the neighborhood of every point $P_j$ the function has an essential singularity: 
$$\psi(u;Q)=e^{k_ju^j}\left(\xi^j_0(u)+\xi^j_1(u) k_j^{-1}+\dots			\right);$$
\item the function is meromorphic outside $P_1+\dots+P_n$ and has simple poles only at the points of the divisor $\gamma$;
\item $\psi(u;R_{\alpha})=d_{\alpha}\in\mathbb{R},\; \alpha = 1,\dots,l+N;$
\end{enumerate}
is called \textit{an $n$-point Baker--Akhiezer function.} 

A Baker--Akhiezer function exists and unique, so it vanishes if the vector $d$ is equal to zero. One can express Baker--Akhiezer function in terms of the $\theta$-function of the curve $\Gamma$ (see \cite{Krichever}).

Consider a curve $\Gamma$ with an holomorphic involution $\sigma:\Gamma\rightarrow\Gamma$, $\sigma^2=id$, such that
\begin{enumerate}
\item the local parameters $k_i$ are odd: $k_i\left(\sigma(P)\right)=-k_i(P)$;
\item the involution $\sigma$ has only $2(n+N)$ fixed points: $P_1,\dots , P_n$, $R_{l+1},\dots, R_{l+N},$ and we denote the remaining points by $Q_1,\dots , Q_{n+N}$.
\end{enumerate}

Choose such the divisors $\gamma$ and $R$ on $\Gamma$ that there is an even (with respect to the involution $\sigma$) meromorphic differential $\Omega$ on $\Gamma$ defined by the divisors in the following way:

\begin{enumerate}
\item it equals to zero at $n+2(g+N+l-1)$ points: $\left(\Omega\right)_0= P_1+\dots+P_n +\gamma + \sigma(\gamma)$; 
\item it has $n+2N+2l$ simple poles: $\left(\Omega\right)_{\infty}= Q_1+\dots+Q_{n+N}+R_1+\dots+R_{l+N}+\sigma(R_1)+\dots+\sigma(R_l)$;
\item all the residues of $\Omega$ are equal to 1: $\Res\limits_{Q_1}\Omega=\dots=\Res\limits_{Q_{n+N}}\Omega=1.$
\end{enumerate}

We denote the residues of $\Omega$ at the points $R_1,\sigma R_1,\dots,R_l,\sigma R_l$ by $r_1,\dots,r_l$:
\begin{equation}
    \label{r-residues}
    r_{\alpha} = \Res\limits_{R_{\alpha}}\Omega = \Res\limits_{\sigma R_{\alpha}}\Omega,\quad \alpha=1,\dots,l.
\end{equation}
The notion is correct since $\Omega$ is even.

Let us introduce reality conditions to get real-valued functions. 

\begin{lemma}

Let $\Gamma$ be a curve with antiholomorphic involution $\tau : \Gamma \rightarrow \Gamma$ such that the points $P_1,\dots,P_n,Q_1,\dots,Q_{n+N}, R_{1}, \dots ,R_{l+N}$ are fixed points of the involution $\tau$, the divisor $\gamma$ is invariant under $\tau$: $\tau(\gamma)=\gamma$, and local parameters are antiholomorphic: $k_i(\tau(P))=\overline{k_i(P)}$.

Then the Baker-Akhiezer function satisfies the relation
$$\psi(u,Q | d)=\overline{\psi(u,\tau(Q)| d)}.$$

\end{lemma}

For the proper choice of algebraic-geometric data that is described above, we have the following theorem.

\begin{theorem}[\cite{Glukhov}]
\label{theorem1}
Functions $x^k(u)=\psi(u, Q_k),$ $k = 1, \dots , n+N$ define a real  $n$-dimensional orthogonal net in $(n+N)$-dimensional Euclidean space.
\end{theorem}

This theorem provides us with an $n$-dimensional orthogonal net $x(u)$ in $\mathbb{R}^{n+N},$ that is constructed according to the algebraic-geometric data $S$. 


\section{Algebraic-Geometry Approach to Ribaucour Transformations}

We choose now $d=(1,\dots,1)$ for the simplicity and obtain the orthogonal net $x(u)$ from the theorem \ref{theorem1}:
$$S\; \rightarrow\; \psi(u,Q)\; \rightarrow x(u).$$

For every $\alpha=1,\dots, l$ the following transformation of algebraic geometrical data 
$$S_{\alpha}=\{\Gamma,\, g;\; \{P_j,\,k_j^{-1}\}_{j=1}^{n};\; R_1,\dots,R_{\alpha-1}, \sigma R_{\alpha}, R_{\alpha+1},\dots, R_{l+N};
\gamma_1,\dots,\gamma_{g+l+N-1}\}$$
defines another orthogonal net $x_{\alpha}(u)$:
$$S_{\alpha}\; \rightarrow\; \psi_{\alpha}(u,Q)\; \rightarrow x_{\alpha}(u).$$

\begin{theorem}
\label{theorem2}
A pair of the orthogonal nets $x(u)$ and $x_{\alpha}(u)$ is a Ribaucour pair.
\end{theorem}
\begin{proof}
To prove the theorem we need two lemmas.

Let us denote the first terms of the expansion of $\psi_{\alpha}(u,Q)$  at $P_j$ by $\xi^j_{0,\,\alpha}(u)$, i.e. the expansion in the neighborhood of $P_j$ has the following form:
$$\psi_{\alpha} (u,Q)=e^{k_ju^j}\left(\xi^j_{0,\,\alpha}(u)+\frac{\xi^j_{1,\,\alpha}(u)}{k_j}+\dots			\right).$$

We consider the function
\begin{equation*}
\Phi_{i,\,\alpha}(u,Q) = \xi^i_0(u) \partial_i \psi_{\alpha} (u, Q) - \xi^i_{0,\,\alpha} (u) \partial_i \psi(u,Q).
\end{equation*}
It is straightforward that $\Phi_{i,\,\alpha}(u,Q)$ is an $n$-point Baker--Akhiezer function with non-constant values at the points of the divisor $R$. The values are 
$$
\Phi_{i,\,\alpha}(u,R_k)=0,\; k\neq\alpha;$$  
$$\Phi_{i,\,\alpha}(u,R_{\alpha})=\xi^i_{0} (u) \partial_i\phi_{\alpha,\,\alpha}(u);\; \Phi_{i,\,\alpha}(u,\sigma R_{\alpha})=-\xi^i_{0,\,\alpha} (u)  \partial_i\phi_{\alpha}(u),$$
where $\phi_{\alpha}(u) = \psi(u,\sigma R_{\alpha})$, $\phi_{\alpha,\alpha}(u)= \psi_{\alpha}(u,R_{\alpha})$.

\begin{lemma}
\label{BA connection lemma}
The function $\Phi_{i,\,\alpha}(u,Q)$ on the curve $\Gamma$ is proportional to the function $\psi_{\alpha}(u,Q) - \psi(u,Q)$:
\begin{equation}
\label{BA connection}
\Phi_{i,\,\alpha}(u,Q) = \xi^i_{0,\alpha}(u) \dfrac{\partial_i \phi_{\alpha}(u)}{\phi_{\alpha}(u) - 1} \left(\psi_{\alpha}(u,Q) - \psi(u,Q)\right).
\end{equation}
\end{lemma}

\begin{proof}
Let us calculate the values of the right side at the points in the divisor $R-R_{\alpha}$ and at the point $\sigma R_{\alpha}$:
$$\xi^i_{0,\alpha}(u) \dfrac{\partial_i \phi_{\alpha}(u)}{\phi_{\alpha}(u) - 1} \left(\psi_{\alpha}(u,R_k) - \psi(u,R_k)\right) = 0,\, k\neq\alpha;$$ 
$$\xi^i_{0,\alpha}(u) \dfrac{\partial_i \phi_{\alpha}(u)}{\phi_{\alpha}(u) - 1} \left(\psi_{\alpha}(u,\sigma R_{\alpha}) - \psi(u,\sigma R_{\alpha})\right) = - \xi^i_{0,\alpha}(u)\partial_i \phi_{\alpha}(u).$$

Thus, the difference of two sides of the equality (\ref{BA connection}) is an $n$-point Baker--Akhiezer function for $S_{\alpha}$ with all the constants $d_i$ equal to zero, which means that the difference equals to zero.
\end{proof}

To express the coefficient on the right side of (\ref{BA connection}) we prove the following lemma. 

\begin{lemma}
\label{coef lemma}
The function $\phi_{\alpha}(u)$ relates to the embedding functions $x(u)$ and $x_{\alpha}(u)$ in the following way
\begin{equation*}
\dfrac{\partial_i \phi_{\alpha}}{\phi_{\alpha}-1}=-2\dfrac{\partial_i x\cdot\delta x}{\delta x\cdot \delta x}.
\end{equation*}
\end{lemma}

\begin{proof}
Let us consider the differential form 
$$\partial_i \psi(u,Q) \left(\psi_{\alpha}(u,\sigma Q)-\psi(u,\sigma Q)\right) \Omega(Q).$$ 
One can show that all the essential singularities of function are not ones for the differential. Then the differential is meromorphic and it has only simple poles $Q_1,\dots,Q_n,\,\sigma R_{\alpha}$.

All the residues of the differential sum up to zero:
$$\sum\limits_{s=1}^{n+N} \partial_i x^s(u) \left(x_{\alpha}^s(u)-x^s(u)\right) + r_{\alpha} \partial_i \phi_{\alpha}(u)\left(\phi_{\alpha,\,\alpha}(u)-1\right)=0,$$
where $r_{\alpha}$ as in (\ref{r-residues}).
We get
\begin{equation}
\label{scalar up}
\partial_i x\cdot \delta x  = - r_{\alpha} \partial_i \phi_{\alpha}\left(\phi_{\alpha,\,\alpha}-1\right).
\end{equation}

Let us consider the differential $$\left(\psi_{\alpha}(u,Q)-\psi(u,Q)\right) \left(\psi_{\alpha}(u,\sigma Q)-\psi(u,\sigma Q)\right) \Omega(Q).$$ 
It is also meromorphic  and has only simple poles $Q_1,\dots,Q_n,\,R_{\alpha},\,\sigma R_{\alpha}$.

All the residues of the differential sum up to zero:
$$\sum\limits_{s=1}^{n+N} \left(x_{\alpha}^s(u)-x^s(u)\right)^2 + 2 r_{\alpha}\left(\phi_{\alpha,\,\alpha}(u)-1\right)\left(1- \phi_{\alpha}(u)\right)=0.$$
We get 
\begin{equation}
\label{scalar down}
\delta x\cdot\delta x  = 2 r_{\alpha} \left(\phi_{\alpha}-1\right)\left(\phi_{\alpha,\,\alpha}-1\right).
\end{equation}

Dividing (\ref{scalar up}) by (\ref{scalar down}) we prove the lemma.
\end{proof}

To prove the main theorem we consider the values of (\ref{BA connection}) at the points $Q_k$, $k=1,\dots,n+N$:
$$\xi_0^i \partial_i x_{\alpha}^k - \xi_{0,\,\alpha}^i \partial_i x^k =\xi^i_{0,\,\alpha} \dfrac{\partial_i \phi_{\alpha}}{\phi_{\alpha}-1} \left(x_{\alpha}^k- x^k\right),\; k = 1,\dots,n+N.$$

In the vector form we have
$$\xi_0^i \partial_ix_{\alpha} - \xi_{0,\,\alpha}^i \partial_i x = \xi_{0,\,\alpha}^i \dfrac{\partial_i \phi_{\alpha}}{\phi_{\alpha}-1} \delta x.$$

Using lemma \ref{coef lemma} we obtain (\ref{RibTrans}) with $\lambda_i(u)=\dfrac{\xi_{0,\,\alpha}^i(u)}{\xi_0^i(u)}$ and prove the theorem.

\end{proof}

If $l>1$, then we construct two different Ribaucour transformation of $x(u)$: $x_{\alpha}(u)$ and $x_{\beta}(u)$ according to transformed algebraic-geometric data $S_{\alpha}$ and $S_{\beta}$ respectively. We transform the transformed algebraic-geometric data and obtain corresponding orthogonal net in the following way: 
$$S_{\alpha\beta}=\left(S_{\alpha}\right)_{\beta}=\left(S_{\beta}\right)_{\alpha}\rightarrow x_{\alpha\beta}(u).$$
Thus, we have Bianchi quadrilaterals:
$$
\begin{matrix}
S_{\beta}& \rightarrow & S_{\alpha\beta} &\qquad& x_{\beta}(u) & \rightarrow   & x_{\alpha\beta}(u)  \cr
 \uparrow &    & \uparrow &\qquad&  \uparrow &   & \uparrow \cr
S    & \rightarrow & S_{\alpha} &\qquad& x(u)   & \rightarrow & x_{\alpha}(u) \cr
\end{matrix}
$$

One can obtain $2^l$ orthogonal nets corresponding to the initial data $S$. Ribaucour pairs are the edges of $l$-cube and form the structure of the cube with 2-faces described above. 

That procedure provides us with an example of a Bianchi $l$-cube. However, the existence of algebraic-geometric data for any Bianchi $l$-cube with the initial orthogonal net $x(u)$ from theorem \ref{theorem1} and $l$ Ribaucour transformations of it from theorem \ref{theorem2} is unclear. This important question requires further investigation.

\section{Examples}

For $l \geq 2$ the explicit formulae are complicated and involve the $\theta$-function of the smooth curve $\Gamma$. For $l=1$ the construction simplifies, and the following theorem describes it .

\begin{theorem}
If $l=1$ and $d_2=\dots=d_{N+1}=0,$ we have the only transformation $x_1(u)$ and
$$x_1(u) = \frac{c}{x(u)\cdot x(u)} x(u),$$ where $c\in\mathbb{R}$ defined by $S$.
\end{theorem}

\begin{proof}
The first step is to prove that 

\begin{equation}
\label{lemma1 l=1}
\psi_1(u,Q) = \frac{1}{\phi_1(u)}\psi(u,Q),
\end{equation}
where $\phi_1(u)=\psi(u,\sigma R_1).$

The function $\frac{1}{\phi_1(u)}\psi(u,Q)$ is a Baker--Akhiezer function and it equals to 1 at the point $\sigma R_1$. Since a Baker--Akhiezer function for data $S_1$ is unique, we have the equality (\ref{lemma1 l=1}).

Thus, the orthogonal nets $x(u)$ and $x_1(u)$ are proportional:
$$x_1(u) = \frac{1}{\phi_1(u)}\ x(u).$$ 

We prove now that 
$$\phi_1(u) = - \dfrac{x(u)\cdot x(u)}{2 r_1}. $$

Consider the differential form $\psi(u,Q) \psi(u,\sigma Q) \Omega(Q)$. It is meromorphic and has only simple poles $Q_1,\dots,Q_n,\,R_1,\,\sigma R_1$. The sum of all the residues vanishes:
$$\sum\limits_{s=1}^{n+N}\left(x^k\right)^2+ 2 r_1 \phi_1=0.$$

That proves the theorem.
\end{proof}

Another way to get explicit examples is to choose a singular curve $\Gamma$ as in \cite{Mironov}. When the curve is reducible and all its irreducible components are rational curves, all the formulae are expressed in elementary functions. For the proper choice of the initial data, one needs to solve a system of linear equations to get $2^l$ orthogonal nets. 

\section*{Acknowledgement}
This work was supported by the Russian Science Foundation under grant no. 18-11-00316.

The author would like to thank scientific advisor Prof. Oleg I. Mokhov for his attention to this work, and helpful and inspirational discussions on the topic.

\bibliographystyle{ieeetr}
\bibliography{Refs}

\end{document}